\documentclass[twoside]{article}
\usepackage[usenames,dvipsnames]{color}
\usepackage{graphicx}
\usepackage{camnum}
\usepackage{harvard}
\usepackage{amsmath,amssymb}
    \usepackage{ulem}
    \usepackage{chngcntr}
\usepackage{mathtools}
\usepackage{bm}
\usepackage{cite}

\usepackage[ruled,vlined]{algorithm2e}

\usepackage{etoolbox}
\usepackage{geometry}[margins=1in]

\usepackage{tikz,pgfplots}
        \pgfplotsset{compat = 1.3}
        \pgfplotsset{minor grid style={dotted}} \pgfplotsset{major grid
        style={dashed}}

        \pgfplotsset{every x tick label/.append style={font=\footnotesize,
        yshift=0.25ex}}

        \pgfplotsset{every y tick label/.append
        style={font=\footnotesize, xshift=0.25ex}}

\definecolor{colorclassyorange}{rgb}{0.95000,0.32500,0.09800}
\definecolor{colorchromeyellow}{rgb}{1.00000,0.6549,0}%
\definecolor{colorpaleyellow}{rgb}{1.00000,0.8549,0.1}%

\definecolor{colorclassyblue}{rgb}{0.00000,0.44706,0.74118}%
\definecolor{colorpurple}{rgb}{0.49400,0.18400,0.55600}%
\definecolor{colorfuschia}{rgb}{0.95039,0.0,0.95039}%
\definecolor{colorlemongreen}{rgb}{0.6,0.8,0}%
\definecolor{colorreal}{rgb}{0.92941,0.79412,0.12549}%
\definecolor{colorimag}{rgb}{0.00000,0.49804,0.00000}%
\definecolor{colorabs}{rgb}{1.00000,0.00000,0.00000}%

\newcommand{\uu}{\MM{u}} 
\newcommand{\uuz}{\MM{u_0}}
\newcommand{\uun}{\MM{u_1}}

\newcommand{\ww}{\MM{w}}
\newcommand{\xx}{\MM{x}}
\newcommand{\UU}{\MM{U}} 
\newcommand{\YY}{\MM{Y}}

\counterwithin{table}{section}
\counterwithin{algocf}{section}
\numberwithin{equation}{section}

\newtoggle{pgfplots}
\togglefalse{pgfplots}

\begin{document}

\title{Exponentially fitted methods that preserve conservation laws}

\author{Dajana Conte, Gianluca Frasca-Caccia\footnote{Corresponding author: gfrascacaccia@unisa.it}}
\date{\normalsize Department of Mathematics, University of Salerno,\\ Via Giovanni Paolo II n. 132, 84084 Fisciano (SA), Italy}
\maketitle

\abstract 
The exponential fitting technique uses information on the expected behaviour of the solution of a differential problem to define accurate and efficient numerical methods. In particular, exponentially fitted methods are very effective when applied to problems with oscillatory solutions. In this cases, compared to standard methods, they have proved to be very accurate even using large integration steps.

In this paper we consider exponentially fitted Runge-Kutta methods and we give characterizations of those that preserve local conservation laws of linear and quadratic quantities.

As benchmark problems we consider wave equations arising as models in several fields such as fluid dynamics and quantum physics, and derive exponentially fitted methods that preserve their conservation laws of mass (or charge) and momentum. The proposed methods are applied to approximate breather wave solutions and are compared to other known methods of the same order.\vspace{.2cm}

\section{Introduction}
We consider a Partial Differential Equation (PDE) in the form
\begin{equation}\label{PDE}
u_t=f(x,t,[u]_x), \qquad (a,b)\times(t_0,T)
\end{equation}
equipped with suitable initial and boundary conditions. Here and henceforth $u(x,t)\in\mathbb{R}$, and given a generic regular function $h$, the symbol $[h]_x$ denotes $h$ and its spatial derivatives. We consider problem (\ref{PDE}) for simplicity of discussion but the arguments can be straightforwardly applied to systems of PDEs, PDEs in multiple dimensions, and PDEs of the form
$$D_t(g(x,[u]_x))=f(x,t,[u]_x),$$
where here and henceforth $D_z$ denotes the total derivative with respect to $z$, and $g$ is linear homogeneous in $[u]_x$. PDEs that arise as a realistic model for a natural phenomenon typically have conservation laws. A conservation law is a total divergence,
\begin{equation}\label{CLaw}
\text{Div\,} \mathbf{F}=D_xF(x,t,[u]_x,[u_t]_x)+D_tG(x,[u]_x),
\end{equation}
that vanishes on solutions of the PDE (\ref{PDE}). The functions $F$ and $G$ are the flux and the density of the conservation law. The density $G$ usually has a physical interpretation such as charge, mass, momentum or energy.

When the boundary conditions are conservative (e.g., periodic) integration in space of (\ref{CLaw}) calculated on solutions of (\ref{PDE}) yields
$$D_t\int_a^b G(x,[u]_x)\,\mathrm{d}x=0,$$
therefore, 
\begin{equation}\label{inv}
\int_a^b G(x,[u]_x)\,\mathrm{d}x
\end{equation} 
is a global invariant.

It is well known that numerical methods that preserve global invariants perform better than standard ones as the accumulation of the error is slower over long times \cite{defrutosSS,DuranSS,Hojj}. On one hand, methods that preserve conservation laws have this same property \cite{NLS}. On the other hand, conservation laws govern the local variation of the conserved quantities, and so a numerical scheme must satisfy stronger constraints to preserve them.

Finite difference methods that preserve conservation laws have been introduced for a range of numerical equations by using a novel technique in \cite{IMA,FoCM,mKdV,NLS}. Although these particular geometric integrators perform better than standard methods, when the solution of the problem is highly oscillatory they require very small stepsizes in order to correctly reproduce the oscillations of the solution.

In this paper we focus on problems whose oscillatory behaviour is known a priori. Problems of this kind are for example breather solutions. In the context of ODEs breathers solutions  are periodic in time with energy localized in only a few low-frequency normal modes \cite{FIK,CEB,CE}. For wave models such as the modified Korteweg-de Vries equation \cite{CJ,MP}, nonlinear Schr\"odinger equation \cite{AEK,TW}, sine-Gordon equation \cite{sGbreath,DX,NLS}, Gardner equation \cite{MP,Sbreath}, breathers are particular solutions that oscillate in time (resp. space) and are localized in space (resp. time) \cite{AKNS}.

When the behaviour of the solution is known a priori, exponential fitting techniques can be used to derive numerical methods that accurately solve the problem at hand with relatively large stepsizes. In fact, exponentially fitted methods are obtained by requiring exact integration of the solutions in a fitting space generated by a suitable set of functions \cite{Ixaru,Pater,DP,CDP}. The choice of the fitting space is made on the basis of the expected behaviour of the solution. For example, when the solution is known to be oscillatory the fitting space can be conveniently defined by combinations of sine and cosine functions that can capture the frequency of oscillation of the solution. In this paper we assume that this frequency is known or that it can be derived from the problem. However, several techniques are available from the literature to numerically estimate the frequency when unknown \cite{DMP,DEP,vDVB,VBMV,VBIM}.

Exponential fitting techniques have been used in several contexts, such as fractional differential equations \cite{BCDP}, quadrature \cite{CPquad,EW,E01,ConteJCAM,CIPS}, interpolation \cite{MVB}, {\em peer} integrators for ODEs and PDEs \cite{Conte1,Conte2}, integral equations \cite{CIP}, boundary value problems \cite{VanD}.

Particular attention has been devoted to the development and analysis of exponentially fitted Runge-Kutta (EFRK) methods (see, e.g., \cite{DEP,DIP,VBIM,PN,DPS,Simos,P1,VVV}), and in particular of symplectic methods (see, e.g., \cite{CFMR,VdV,vDVB,TVA}) that have extended to the context of exponential fitting the well known theory of Runge-Kutta methods \cite{HLW,SSC}. 

In this spirit it has been proved in \cite{CFMR} that symplectic EFRK methods conserve all linear and quadratic invariants of a system of ODEs. This extends the analogue result that is well known to hold for symplectic Runge-Kutta methods \cite{CIZ}.

In the context of PDEs symplectic Runge-Kutta methods preserve all the conservation laws with either linear or quadratic density $G$ of a suitable space discretization \cite{FoCM}. 

The scope of this paper is twofold. On one hand, we prove that symplectic EFRK methods satisfy a similar property. However, only some of them preserve conservation laws with quadratic non homogeneous density $G$. On the other hand, we show the benefits of using symplectic EFRK methods to solve PDEs with breather type solutions. 

The paper is organized as follows. In Section~\ref{sec:space} we show how to use the technique in \cite{FoCM} to define space discretizations that preserve conservation laws. In Section~\ref{sec:time} we deal with the time discretization and prove original results on the conservation properties of EFRK methods. In Section~\ref{sec:methods} we consider three different equations: the linear advection equation, the modified KdV equation and the nonlinear Schr\"odinger equation. The latter provides an application of the new theory in this paper to a system of PDEs. For each of these three PDEs we introduce a conservative exponentially fitted numerical method and give explicit formulae of the discrete conservation laws satisfied by the numerical solutions. In Section~\ref{sec:numerics} we consider oscillatory and breather solutions and use them as benchmark problems to show the conservative properties of the proposed methods and their advantages compared to classic symplectic methods of the same order. Finally, in Section~\ref{sec:concl} we draw some conclusive remarks.
\section{Space discretization}\label{sec:space}
In this section we show how to find finite difference semidiscretizations of (\ref{PDE}) that preserve conservation laws. We first select $q$ conservation laws to preserve and write them in characteristic form,
\begin{equation}\label{char}
D_xF_{\ell}(x,t,[u]_x,[u_t]_x)+D_tG_{\ell}(x,[u]_x)=\mathcal{A}\mathcal{Q}_\ell,\qquad \ell=1,\ldots,q,
\end{equation}
where 
$$ \mathcal{A}=u_t-f(x,t,[u]_x),$$
and $\mathcal{Q}_\ell$ is a multiplier function  called the characteristic of the $\ell$-th conservation law to preserve, that may depend on $x$, $t$, $u$ and its partial derivatives \cite{Olver}.

Given a uniform grid of nodes, $$x_m=a+(m-1)\Delta x,\qquad m=1,\ldots,M,\qquad \Delta x=\frac{b-a}{M-1},$$ 
we define the vectors
$$\xx\in\mathbb{R}^{M},\qquad \xx_m=x_m,\qquad \UU=\UU(t)\in\mathbb{R}^{M},\qquad \UU_m(t)\simeq u(x_m,t),\qquad m=1,\ldots,M,$$
the forward shift,
$$S_{\Delta x}(f(x_m,t,U_m))=f(x_{m+1},t,U_{m+1}),$$
and the forward difference operator
$$D_{\Delta x}=\frac{S_{\Delta x}-I}{\Delta x},$$
where $f$ is a generic function defined on the grid, and $I$ is the identity operator. We look for a space discretization of (\ref{PDE}),
\begin{equation}\label{SD}
\widetilde{\mathcal{A}}(\xx,t,\UU):=D_t{\UU}-\widetilde{f}(\xx,t,\UU)=0,
\end{equation}
and for approximations $\widetilde{\mathcal{Q}}_\ell$ of $\mathcal{Q}_\ell$ such that 
\begin{equation}\label{SDCLaw}
\widetilde{\mathcal{A}}\widetilde{\mathcal{Q}}_\ell=D_{\Delta x}\widetilde F_\ell(\xx,t,\UU,\UU_t)+D_t\widetilde G_\ell(\xx,\UU),\qquad \ell=1,\ldots,q,
\end{equation}
with $\widetilde{F}_\ell\approx F_\ell$ and $\widetilde{G}_\ell\approx G_\ell$. The following theorem gives a characterization of the space of semidiscrete divergences,
\begin{equation}\label{eq:SDdiv}
D_{\Delta x}\widetilde F(\xx,t,\UU,\UU_t)+D_t\widetilde G(\xx,\UU),
\end{equation}
\begin{theorem}[\cite{FoCM}]\label{theo:euler}
The kernel of the semidiscrete Euler operator
$$\mathsf{E}=\sum_{i,j}S_{\Delta x}^{-i}(-D_t)^j\frac{\partial}{\partial(D_t^jU_i^\alpha)},$$
is the space of semidiscrete divergences (\ref{eq:SDdiv}).
\end{theorem}
On the basis of this result, the following strategy is here used to find bespoke finite difference space discretizations (\ref{SD}) with conservation laws (\ref{SDCLaw}).

\begin{enumerate}
\item Select a set of nodes and define on it generic finite difference approximations $\widetilde{\mathcal{A}}\approx\mathcal{A}$ and $\widetilde{\mathcal{Q}}_1\approx\mathcal{Q}_1$. The coefficients defining these approximations are free parameters to be determined.
\item Fix some of the parameters by solving desired order conditions.
\item Fix more parameters by solving symbolically
\begin{equation}\label{eq:sdEuler}
\mathsf{E}(\mathcal{\widetilde A \widetilde Q}_1)=0.
\end{equation}
As a consequence of Theorem \ref{theo:euler} there exist $\widetilde{F}_1\approx F_1$ and $\widetilde{G}_1\approx G_1$ such that 
$$\widetilde{\mathcal A}\widetilde{\mathcal Q}_1=D_{\Delta x} \widetilde{F}_1 + D_t \widetilde{G}_1$$
is a conservation law of $\widetilde{\mathcal A}$.
\item Iterate to preserve more conservation laws replacing $\widetilde{\mathcal Q}_1$ with $\widetilde{\mathcal Q}_\ell$ until equation (\ref{eq:sdEuler}) admits solutions.
\end{enumerate}
Typically, at the end of the procedure above one obtains a family of methods that depend on some remaining parameters that can be arbitrarily chosen \cite{FoCM,IMA,mKdV,NLS}. In general, optimal values of these parameters are not available a priori. However, these can be identified by minimizing an estimate of the local truncation error for the specific problem at hand, as done in \cite{Singh}. 

The purpose of this paper is neither finding all possible semidiscretizations with the desired conservation laws, nor identifying the most accurate of these schemes for a given problem. Therefore, in the following we set equal zero any remaining parameter that can be arbitrarily chosen at the end of the procedure above. Hence, for each of the equations studied in this paper, we focus on only one of infinitely many possible conservative semidiscretizations.
\section{Exponentially fitted time integration}\label{sec:time}
For the time integration of (\ref{PDE}) we consider here only one-step methods and we discuss only the first step of integration that can be similarly iterated.

Let be $\uuz$ the vector of the values of the initial condition at the nodes $\xx$, and $\uun$ the approximation at the next time step,
$$t_1=t_0+\Delta t,$$
given by an $s$-stage EFRK method. An $s$-stage EFRK method applied to (\ref{SD}) amounts to
\begin{align}\label{EFRKsol}
\uun &\,= \uuz  + \Delta t\sum_{i=1}^s
b_i\widetilde{f} (\xx,t_0+c_i\Delta t , \YY_i ),\\\label{EFRKstages}
\YY_i &\,= \gamma_i \uuz + \Delta t \sum_{j=1}^s a_{i,j}\widetilde{f} (\xx,t_0+c_j\Delta t, \YY_j),\qquad i = 1,...,s,
\end{align}
where the real coefficients $\gamma_i>0,$ $a_{i,j},$ $b_i,$ and $c_i$ may depend on the time step, $\Delta t$, and on a parameter, $\omega$, that characterizes the exact solution \cite{Ixaru}. These coefficients are obtained by choosing a fitting space and requiring that (\ref{EFRKsol})--(\ref{EFRKstages}) is exact on solutions that belong to the fitting space. The choice of the fitting space is based on the expected behaviour of the solution. For example, when the solution is oscillatory, the parameter $\omega >0$ is the frequency of oscillation, and the basis of the fitting space is typically chosen as
$$\mathcal{F}=\{1,x,\ldots,x^K,\cos(\omega t),\sin(\omega t),x\cos(\omega t),x\sin(\omega t),\ldots,x^P\cos(\omega t),x^P\sin(\omega t)\}.$$  
The coefficients defining method (\ref{EFRKsol})--(\ref{EFRKstages}) are then obtained by requiring that the functionals
\begin{align}\label{eq:Lsol}
\mathcal{L}[y(t),\Delta t]:=&\,y(t+\Delta t) - y(t) - \Delta t\sum_{i=1}^s
b_i y'(t+c_i\Delta t),\\\label{eq:Lstages}
\mathcal{L}_i[y(t),\Delta t]:=&\,y(t+c_i\Delta t) - y(t) - \Delta t \sum_{j=1}^s a_{i,j}y'(t+c_j\Delta t),\qquad i = 1,...,s,
\end{align}
all vanish for any function $y$ in the fitting space generated by $\mathcal{F}$.

If, moreover, the coefficients $a_{i,j},b_i$ and $\gamma_i$ satisfy 
\begin{equation}\label{symplcond}
b_i\frac{a_{i,j}}{\gamma_i}+b_j\frac{a_{j,i}}{\gamma_j}-b_ib_j=0,\qquad 1\leq i,\,j\leq s,
\end{equation}
then the EFRK method (\ref{EFRKsol})--(\ref{EFRKstages}) is symplectic \cite{VdV}.

It is known that symplectic EFRK methods conserve all linear and quadratic invariants of a system of ODEs \cite{CFMR}. We prove here some local conservation properties of symplectic EFRK methods when applied to a space discretization (\ref{SD}) of a PDE (\ref{PDE}) with conservation laws (\ref{SDCLaw}). For clarity of notation, henceforth we drop the index $\ell$ in (\ref{SDCLaw}) and refer to a generic conservation law, 
\begin{equation}\label{eq:SDCLgen}
D_{\Delta x}\widetilde{F}(\xx,t,\UU,\UU_t)+D_t\widetilde{G}(\xx,\UU)=0,
\end{equation}
satisfied by the solutions of the semidiscretization (\ref{SD}).
\begin{theorem}\label{THnew}
The EFRK method (\ref{EFRKsol})–-(\ref{EFRKstages})
\begin{enumerate}
\item preserves the conservation law (\ref{eq:SDCLgen}) with $\widetilde{G}$ linear in $\UU$.
\item preserves the conservation laws (\ref{eq:SDCLgen}) with $\widetilde{G}$ quadratic homogeneous in $\UU$ iff its coefficients satisfy condition (\ref{symplcond}).
\end{enumerate}
\end{theorem}
\begin{proof}
\begin{enumerate}
\item As $\widetilde{G}(\xx,\UU)$ is linear in $\UU$, its $m$-th entry, $\widetilde{G}_m$, relative to the node $x_m$ is in the form
$$\widetilde{G}_m(\xx,\UU)=\ww_m^T(\xx)\UU,\qquad  \ww_m\in\mathbb{R}^M.$$
Differentiating (\ref{eq:SDCLgen}) and substituting, yields
\begin{equation}\label{DmF}
D_{\Delta x}\widetilde{F}_m(\xx,t,\UU,\widetilde{f}(\xx,t,\UU))=-\ww_m^T(\xx)\widetilde{f}(\xx,t,\UU),
\end{equation}
where $\widetilde{F}_m$ is the $m$-th entry of $\widetilde{F}(\xx,\UU)$ relative to $x_m$.
Multiplying (\ref{EFRKsol}) by $\ww_m(\xx)^T$ on the left gives
\begin{equation}\label{EFRKsolsproof}
\widetilde{G}_m(\xx,\uun) = \widetilde{G}_m(\xx,\uuz)  + \Delta t\sum_{i=1}^s
b_i\ww_m(\xx)^T\widetilde{f}^i,
\end{equation}
where here and henceforth, $$\widetilde{f}^i:= \widetilde{f}(\xx,t_0+c_i\Delta t , \YY_i ).$$
Considering (\ref{DmF}), equation (\ref{EFRKsolsproof}) shows that the solutions of (\ref{EFRKsol})--(\ref{EFRKstages}) satisfy the totally discrete conservation law
\begin{equation}\label{LinCLaw}
D_{\Delta t}(\widetilde{G}_m(\xx,\uuz))+D_{\Delta x}\left(\sum_{i=1}^sb_i \widetilde{F}_m(\xx,t_0+c_i\Delta t , \YY_i, \widetilde{f}(\xx,t_0+c_i\Delta t , \YY_i ))\right)=0,
\end{equation}
where $D_{\Delta t}$ is the forward difference operator in time.
\item As $\widetilde{G}$ is quadratic homogeneous, then there exists $\Omega_m(\xx)=\Omega_m^T(\xx)\in\mathbb{R}^{M\times M}$ such that
$$\widetilde{G}_m(\xx,\UU)=\tfrac{1}2\UU^T\Omega_m(\xx)\UU.$$
The conservation law (\ref{eq:SDCLgen}) yields
\begin{equation}\label{SDCLawproof}
D_{\Delta x}\widetilde F_m(\xx,t,\UU,\widetilde{f}(\xx,t,\UU))=-D_t\widetilde G_m(\xx,\UU)=-\UU^T\Omega_m(\xx)\widetilde{f}(\xx,t,\UU).
\end{equation}
Therefore, considering (\ref{EFRKsol}),
\begin{align}\label{step1}
\widetilde{G}_m(\xx,\uun)&\,=\tfrac{1}2( \uuz  + \Delta t\sum_{i=1}^s
b_i\widetilde{f}^i)^T\Omega_m(\xx)( \uuz  + \Delta t\sum_{i=1}^s
b_i\widetilde{f}^i)\\\nonumber
&\,=\tfrac{1}2\uuz^T\Omega_m(\xx)\uuz+\Delta t\sum_{i=1}^sb_i\uuz^T\Omega_m(\xx)\widetilde{f}^i+\tfrac{\Delta t^2}2\sum_{i,j=1}^s b_ib_j\widetilde{f}^i\,^T\Omega_m(\xx)\widetilde{f}^j.
\end{align}
Moreover, equation (\ref{EFRKstages}) gives
$$\uuz=\frac{\YY_i}{\gamma_i}-\Delta t\sum_{j=1}^s\frac{a_{i,j}}{\gamma_i}\widetilde{f}^j,\qquad i=1,\ldots,s.$$
Substituting in the first sum in (\ref{step1}), yields
\begin{align*}
\widetilde{G}_m(\xx,\uun)\!=\!\widetilde{G}_m(\xx,\uuz)\!+\Delta t\sum_{i=1}^s\frac{b_i}{\gamma_i}{\YY_i^T}\Omega_m\widetilde{f}^i+\!\tfrac{\Delta t^2}2\!\!\sum_{i,j=1}^s\!\!\left(b_ib_j-b_i\frac{a_{i,j}}{\gamma_i}-b_j\frac{a_{j,i}}{\gamma_j}\right)\!{\widetilde{f}^i}\,^T\Omega_m\widetilde{f}^j\!.
\end{align*}
Therefore, taking into account (\ref{symplcond}) and (\ref{SDCLawproof}),
\begin{equation}\label{Dclaw}
D_{\Delta t}\widetilde{G}_m(\xx,\uuz)+D_{\Delta x}\left(\sum_{i=1}^s\frac{b_i}{\gamma_i}\widetilde{F}_m\left(\xx,t_0+c_i\Delta t,\YY_i,\widetilde{f}(\xx,t_0+c_i\Delta t , \YY_i )\right)\right)
\end{equation}
is a totally discrete approximation of (\ref{eq:SDCLgen}) satisfied by the solutions of (\ref{EFRKsol})--(\ref{EFRKstages}) applied to (\ref{SD}).

\end{enumerate}
\end{proof}

In literature several authors (see, e.g., \cite{VdV,CFMR,DIP,Simos}) have focused on EFRK methods (\ref{EFRKsol})--(\ref{EFRKstages}) with
\begin{equation}\label{gammaeq}
\gamma_i=\gamma,\qquad i=1,\ldots,s.
\end{equation}
For these particular schemes the following more general result holds true.

\begin{theorem}\label{THgen}
The EFRK method (\ref{EFRKsol})–-(\ref{EFRKstages}) with coefficients (\ref{gammaeq}) preserves (\ref{eq:SDCLgen}) with $\widetilde{G}$ quadratic in $\UU$ iff its coefficients satisfy (\ref{symplcond}).
\end{theorem}
\begin{proof}

The $m$-th entry of the density $\widetilde{G}(\xx,\UU)$ is of the form 
\begin{equation*}
\widetilde{G}_m(\xx,\UU)=\tfrac{1}2\UU^T\Omega_m(\xx)\UU+\ww_m^T(\xx)\UU,
\end{equation*}
where $\Omega_m$ and $\ww_m$ are defined as above. The semidiscrete conservation law (\ref{eq:SDCLgen}) at the point $x_m$ amounts to
\begin{equation}\label{DmF2}
D_{\Delta x}\widetilde F_m(\xx,t,\UU,\widetilde{f}(\xx,t,\UU))=-(\UU^T\Omega_m(\xx)+\ww^T_m(\xx))\widetilde{f}(\xx,t,\UU).
\end{equation}
Therefore, (\ref{EFRKsol}) yields
\begin{align}\nonumber
\widetilde{G}_m(\xx,\gamma\uun)=&\left(\tfrac{\gamma}2\uuz^T\Omega_m(\xx)+\ww_m^T(\xx)\right)\gamma\uuz+\gamma\Delta t\sum_{i=1}^sb_i(\gamma\uuz^T\Omega_m(\xx)+\ww_m^T(\xx))\widetilde{f}^i\\\nonumber
&+\tfrac{\gamma^2\Delta t^2}2\sum_{i,j=1}^s b_ib_j\widetilde{f}^i\,^T\Omega_m(\xx)\widetilde{f}^j.
\end{align}
Substituting
$$\gamma\uuz={\YY_i}-\Delta t\sum_{j=1}^s{a_{i,j}}\widetilde{f}^j,\qquad i=1,\ldots,s,$$
in the first sum, yields
\begin{align*}
\widetilde{G}_m(\xx,\gamma\uun)&\,=\widetilde{G}_m(\xx,\gamma\uuz)+\gamma\Delta t\sum_{i=1}^sb_i\left({\YY_i^T}\Omega_m(\xx)+\ww_m(\xx)^T\right)\widetilde{f}^i\\
&+\tfrac{\gamma^2\Delta t^2}2\sum_{i,j=1}^s\left(b_ib_j-b_i\frac{a_{i,j}}{\gamma}-b_j\frac{a_{j,i}}{\gamma}\right){\widetilde{f}^i}\,^T\Omega_m(\xx)\widetilde{f}^j.
\end{align*}
Therefore, taking into account (\ref{DmF2}) and (\ref{symplcond}), the method preserves the conservation law
\begin{equation}\label{Dclawgen}
D_{\Delta t}\widetilde{G}_m(\xx,\gamma\uuz)+D_{\Delta x}\left(\gamma\sum_{i=1}^sb_i\widetilde{F}_m\left(\xx,t_0+c_i\Delta t,{\YY_i},\widetilde{f}(\xx,t_0+c_i\Delta t , \YY_i )\right)\right)=0.
\end{equation}
\end{proof}
\subsection*{One-step methods}
We focus now on the simple case of one-stage EFRK methods, obtained by requiring the exactness for solutions in the fitting space generated by
$$\mathcal{F}=\langle\exp{(\pm\mathrm{i}\omega t)}\rangle=\langle\cos{(\omega t)},\sin{(\omega t)}\rangle.$$
The coefficients of the method are obtained by solving
\begin{align}\label{eq:sysmid}
\mathcal{L}[\cos(\omega t),\Delta t]=\mathcal{L}[\sin(\omega t),\Delta t]=\mathcal{L}_1[\cos(\omega t),\Delta t]=\mathcal{L}_1[\sin(\omega t),\Delta t]=0,
\end{align}
where the functionals $\mathcal{L}$ and $\mathcal{L}_1$ are defined as in (\ref{eq:Lsol}) and (\ref{eq:Lstages}), respectively, with $s=1$.

The solution of (\ref{eq:sysmid}) in dependence of $c_1=c_1(\omega, \Delta t)$ and $\nu=\omega \Delta t$, is \cite{CFMR}
$$\gamma_1=\frac{1}{\cos(c_1\nu )},\qquad a_{1,1}=\frac{\tan(c_1\nu )}{\nu},\qquad b_1=\frac{\sin{\nu}}{\nu\cos( c_1\nu)}.$$
Moreover, condition (\ref{symplcond}) amounts to
$$b_1=2\gamma_1^{-1}a_{1,1},$$
and yields
$$2\sin(c_1\nu)\cos(c_1\nu)=\sin \nu,$$
that is satisfied iff $c_1=\tfrac{1}2$. Therefore, there is a unique one-stage EFRK method that preserves discrete versions of conservation laws with quadratic density. This is the exponentially fitted midpoint (EF midpoint) method and amounts to
\begin{align*}
\uun=&\,\uuz+\Delta t b_1 \widetilde f(\xx,t_0+\tfrac{1}2\Delta t, \YY_1),\\
\YY_1=&\,\gamma_1\uuz+\Delta ta_{1,1}\widetilde{f}(\xx,t_0+\tfrac{1}2\Delta t, \YY_1).
\end{align*}
Considering that $$\frac{a_{1,1}}{b_1}=\frac{\gamma_1}{2}\qquad \Longrightarrow\qquad  \YY_1=\frac{\gamma_1}{2}(\uun+\uuz),$$
the method can be equivalently recast in the more compact form,
\begin{equation}\label{EFMP}
\uun=\uuz+\frac{\Delta t\sin{\nu}}{\nu \cos(\nu/2)}\widetilde{f}\left(\xx,t_0+\frac{\Delta t}2,\frac{\uun+\uuz}{2\cos(\nu/2)}\right),
\end{equation}
that converges to the classic midpoint method when $\nu\rightarrow 0$.

As the method is one-step, then Theorem \ref{THgen} applies. Hence, if $\widetilde{G}$ is a quadratic density, the discrete conservation law (\ref{Dclawgen}) satisfied by the solutions of (\ref{EFMP}) at the node $x_m$ amounts to
\begin{align}\label{MPclaw}
&D_{\Delta t}\widetilde{G}_m\!\left(\!\xx,\frac{\uuz}{\cos(\nu/2)}\right)\!+\!D_{\Delta x}\!\left(\frac{\sin{\nu}}{\nu \cos^2(\nu/2) }\widetilde{F}_m\!\left(\!\xx,t_0\!+\!\frac{\Delta t}2,\frac{\uun+\uuz}{2\cos(\nu/2)},\frac{\nu \cos(\nu/2)}{\sin{\nu}}D_{\Delta t}\uuz\right)\!\right)\!=0.
\end{align}
Note that in the particular cases when $\widetilde{G}$ is linear or quadratic homogeneous (\ref{MPclaw}) reduces to (\ref{LinCLaw}) or (\ref{Dclaw}), respectively, with $s=1$.

\section{Conservative exponentially fitted schemes}\label{sec:methods}
In this section we consider three different PDEs: the linear advection equation, the modified KdV (mKdV) equation, and the nonlinear Schr\"odinger (NLS) equation. Each of these equations has two conservation laws whose density $G$ is either linear or quadratic. They govern the local variation of mass and charge or momentum, respectively. 

For each equation we introduce here an exponentially fitted method that preserves the selected conservation laws locally. The conservative semidiscretizations are found following the strategy described in Section \ref{sec:space}. Fully discrete second order schemes are then obtained by applying the EF midpoint method (\ref{EFMP}) in time.

Henceforth we denote with $\mu_{\Delta x}$ the forward average operator in space whose action on a function $f$ of the semidiscrete or discrete solution is defined as
$$\mu_{\Delta x}:f({U}_m)\rightarrow \frac{f({U}_{m+1})+f({U}_{m})}2,\qquad \mu_{\Delta x}:f({u}_{m,n})\rightarrow \frac{f({u}_{m+1,n})+f({u}_{m,n})}2,$$
respectively. The discrete forward average operator in time is analogously defined as
$$\mu_{\Delta t}:f({u}_{m,n})\rightarrow \frac{f({u}_{m,n+1})+f({u}_{m,n})}2.$$
\subsubsection*{Linear advection equation}
The linear advection equation, 
\begin{equation}\label{LAE}
u_t=\omega u_x,
\end{equation}
is itself a conservation law, with density, flux and characteristic
$$F_1=-\omega u,\qquad G_1=u\qquad \mathcal{Q}_1=1$$
Solutions of (\ref{LAE}) satisfy the momentum conservation law defined by
$$F_2=-\frac{\omega}2 {u^2},\qquad G_2=\frac{1}2u^2\qquad \mathcal{Q}_2=u.$$
Centred at a generic point $x_m$, the semidiscretization
\begin{equation}\label{SDLAE}
D_t {U}_m=\omega D_{\Delta x}\mu_{\Delta x} {U}_{m-1},
\end{equation}
has a semidiscrete mass conservation law defined by 
$$\widetilde F_1=-\omega\mu_{\Delta x} {U}_{m-1},\qquad \widetilde G_1=U_m,\qquad \widetilde{\mathcal{Q}}_1=1,$$
and a momentum conservation law with flux, density and characteristic
$$\widetilde F_2=-\frac{\omega}2  {U}_{m}{U}_{m-1},\qquad \widetilde G_2=\frac{1}2{U}_m^2,\qquad \widetilde{\mathcal{Q}}_1=U_m,$$
respectively. Applying the EF midpoint method (\ref{EFMP}) to (\ref{SDLAE}) gives a fully discrete scheme. The discrete conservation laws are given by equation (\ref{MPclaw}) and amount to
\begin{align}\label{CLLAE1}
&D_{\Delta t}{u}_{m,0}+D_{\Delta x}\left(-\frac{\omega\sin \nu}{\nu \cos^2(\nu/2)}\mu_{\Delta x} (\mu_{\Delta t} u_{m-1,0})\right)=0,\\\label{CLLAE2}
&D_{\Delta t} \left(\frac{1}2u_{m,0}^2\right)+D_{\Delta x}\left(-\frac{\omega\sin \nu}{2\nu \cos^2(\nu/2)}(\mu_{\Delta t} {u}_{m-1,0})(\mu_{\Delta t} {u}_{m,0})\right)=0.
\end{align}
\subsubsection*{Modified Korteweg-de Vries equation}
The mKdV equation,
\begin{equation}\label{mKdV}
u_t+6u^2u_x+u_{xxx}=0,
\end{equation}
 has infinitely many conservation laws. Among them the ones of the mass and momentum are defined by
$$F_1=2u^3+u_{xx},\qquad G_1=u,\qquad \mathcal{Q}_1=1,$$
and
$$F_2=\tfrac{3}2u^4+uu_{xx}-\tfrac{1}2u_x^2,\qquad G_2=\tfrac{1}2u^2, \qquad \mathcal{Q}_2=u,$$
respectively. Fully discrete schemes for the mKdV equation that preserve these conservation laws have been introduced in \cite{mKdV}. Considering a generic point $x_m$, the semidiscretization
\begin{equation}\label{SDmKdV}
D_t {U}_m+D_{\Delta x}\left(2(\mu_{\Delta x}U_{m-1})\mu_{\Delta x}(U_{m-1}^2)+D_{\Delta x}^2\mu_{\Delta x}U_{m-2}\right)=0,
\end{equation}
has mass and momentum conservation laws defined by
\begin{align*}
\widetilde{F}_1=&\,2(\mu_{\Delta x}U_{m-1})\mu_{\Delta x}(U_{m-1}^2)+D_{\Delta x}^2\mu_{\Delta x}U_{m-2},\qquad \widetilde{G}_1=U_m,\qquad \widetilde{\mathcal Q}_1=1,\\
\widetilde{F}_2=&\,\tfrac{1}{2}U_{m-1}U_m(U_{m-1}^2\!+\!U_m^2\!+\!U_{m-1}U_m)\!+\!(\mu_{\Delta x}U_{m-1})D_{\Delta x}^2\mu_{\Delta x}U_{m-2}\\
&-\tfrac{1}4(D_{\Delta x}U_{m-1})D_{\Delta x}(U_{m-2}\!+\!U_m), \qquad\widetilde{G}_2=\tfrac{1}2U_m^2,\qquad \widetilde{\mathcal{Q}}_2=U_m.
\end{align*}
The fully discrete scheme obtained by applying the EF midpoint method (\ref{EFMP}) to (\ref{SDmKdV}) has totally discrete conservation laws given by (see (\ref{MPclaw})),
\begin{align}\label{eq:mKdVcl1}
&D_{\Delta t} u_{m,0}+D_{\Delta x}\left(\frac{\sin \nu}{\nu\cos(\nu/2)}\widetilde{F}_1\vert_{U_m={(\mu_{\Delta t} u_{m,0})}/{\cos(\nu/2)}}\right),\\\label{eq:mKdVcl2}
&D_{\Delta t} u_{m,0}^2+D_{\Delta x}\left(\frac{\sin \nu}{\nu}\widetilde{F}_2\vert_{U_m={(\mu_{\Delta t} u_{m,0})}/{\cos(\nu/2)}}\right)
\end{align}
\subsubsection*{The nonlinear Schr\"odinger equation}
We finally consider the NLS equation for a complex solution, $\psi(x,t)\in\mathbb{C}$,
$$i\psi_t+\psi_{xx}+|\psi|^2\psi=0.$$
Setting $\psi=u+iv$, with $u,v\in\mathbb{R}$, the NLS equation can be equivalently written as the following system of PDEs for its real and imaginary part:
\begin{align}\label{NLS1}
\mathcal{A}:=\left(u_t+v_{xx}+(u^2+v^2)v,
-v_t+u_{xx}+(u^2+v^2)u\right)={0}.
\end{align}
Also the NLS equation has infinitely many conservation laws. Among them, the ones defined by
\begin{align*}
F_1=2uv_x-2u_xv,\qquad G_1=u^2+v^2,\qquad \mathcal{Q}_1=(2u,-2v)^T,
\end{align*}
and by
\begin{align*}
F_2=u_x^2+v_x^2+u_tv-uv_t+\tfrac{1}2(u^2+v^2)^2,\qquad G_2=uv_x-u_xv, \qquad \mathcal{Q}_2=(2v_x,2u_x)^T,
\end{align*}
represent the conservation laws of charge and momentum, respectively. Although in Section \ref{sec:space} we have only considered a single PDE, the strategy can be straightforwardly extended to deal with systems of PDEs (see \cite{FoCM}). A wide range of numerical methods for the NLS equation that preserve both the charge and momentum conservation laws has been derived in \cite{NLS}. The semidiscretization,
\begin{align}\label{SDNLS1}
D_t U_m+D_{\Delta x}^2V_{m-1}+\frac{1}2(U_0(U_{-1}+U_1)+V_0(V_{-1}+V_1))V_0=&\,0,\\\label{SDNLS2}
-D_t V_m+D_{\Delta x}^2U_{m-1}+\frac{1}2(U_0(U_{-1}+U_1)+V_0(V_{-1}+V_1))U_0=&\,0.
\end{align}
has conservation laws of charge and momentum defined by
\begin{align*}
\widetilde{F}_1=\,&2(\mu_{\Delta x}U_{m-1})D_{\Delta x}V_{m-1}\!-\!2(D_{\Delta x}U_{m-1})\mu_{\Delta x}V_{m-1},\quad \widetilde{G}_1\!=\!U_m^2\!+\!V_m^2,\quad \widetilde{\mathcal Q}_1\!=(2U_m,\!-2V_m)^T\!\!,\\ 
\widetilde{F}_2=\,&(D_{\Delta x}U_{m-1})^2+(D_{\Delta x}V_{m-1})^2+(\mu_{\Delta x}V_{m-1})(D_t\mu_{\Delta x} U_{m-1})-(\mu_{\Delta x}U_{m-1})(D_t\mu_{\Delta x} V_{m-1})\\
&+\tfrac{1}2(U_mU_{m-1}+V_mV_{m-1})^2-\!\tfrac{\Delta x^2}4\{(D_tD_{\Delta x}U_{m-1})D_{\Delta x}V_{m-1}\!-\!(D_tD_{\Delta x}V_{m-1})D_{\Delta x}U_{m-1}\}, \\
\widetilde{G}_2=\,&U_mD_{\Delta x}\mu_{\Delta x}V_{m-1}\!-\!V_mD_{\Delta x}\mu_{\Delta x}U_{m-1},\qquad \widetilde{\mathcal{Q}}_2=(2D_{\Delta x}\mu_{\Delta x}V_{m-1},2D_{\Delta x}\mu_{\Delta x}U_{m-1})^T,
\end{align*}
respectively. The conservation laws preserved by the EF midpoint, obtained from (\ref{MPclaw}), are:
\begin{align}\label{eq:NLScl1}
&D_{\Delta t} (u_{m,0}^2+v_{m,0}^2)+D_{\Delta x}\left(\frac{\sin \nu}{\nu}\widetilde{F}_1\right),\\\label{eq:NLScl2}
&D_{\Delta t} (u_{m,0}D_{\Delta x}\mu_{\Delta x}v_{m-1,0}-v_{m,0}D_{\Delta x}\mu_{\Delta x}u_{m-1,0})+D_{\Delta x}\left(\frac{\sin \nu}{\nu}\widetilde{F}_2\right),
\end{align}
where $\widetilde{F}_1$ and $\widetilde{F}_2$ are calculated at $U_m={(\mu_{\Delta t} u_{m,0})}/{\cos(\nu/2)},$ and $V_m={(\mu_{\Delta t} v_{m,0})}/{\cos(\nu/2)}.$
\section{Numerical Tests}\label{sec:numerics}
In this section we propose some numerical tests that highlight the advantages of exponentially fitted methods over classic integrators when the solution of the problem is oscillatory, together with their conservative properties proved in this manuscript. As in \cite{IMA} we evaluate the error on the $q$-th conservation law either as
\begin{equation}\label{CLerr1}
{\rm Err}_q=\Delta x\max_{n}\left|\sum_m \left(\widetilde{G}_q\big\vert_{(x_m,t_n)}-\widetilde{G}_q\big\vert_{(x_m,t_0)}\right)\right|,\qquad q=1,2,
\end{equation}
if the boundary conditions are periodic or zero, or as
\begin{equation}\label{CLerr2}
{\rm Err}_q=\max_{m,n}\left(\{D_{\Delta x}\widetilde{F}_q+D_{\Delta t}\widetilde{G}_q\}\big\vert_{(x_m,t_n)}\right),\qquad q=1,2,
\end{equation}
otherwise. Note that (\ref{CLerr1}) is an estimate of the error in the conservation of the invariant (\ref{inv}). The error in the solution is calculated as
$${\rm Sol\,\, err}=\frac{\|\uu_N-u_{\rm exact}(\xx,T)\|}{\|u_{\rm exact}(\xx,T)\|}$$
for equations (\ref{LAE}) and (\ref{mKdV}), and as
$${\rm Sol\,\, err}=\sqrt{\frac{\|\uu_N-u_{\rm exact}(\xx,T)\|^2+\|\mathbf{v}_N-v_{\rm exact}(\xx,T)\|}{\|u_{\rm exact}(\xx,T)\|^2+\|v_{\rm exact}(\xx,T)\|^2}}$$
for system (\ref{NLS1}), where $N$ is the last time step. For small $\Delta x$, the order of accuracy of the time integrator is estimated as
$${\rm Order}=\frac{\log({\rm Sol\,\, err}_{\Delta t_1}/{\rm Sol\,\, err}_{\Delta t_2})}{\log(\Delta t_1/\Delta t_2)},$$
where the error in the solution is evaluated as above using two different stepsizes, $\Delta t_1$ and $\Delta t_2$.
\subsubsection*{Linear advection equation}
The solution of the linear advection equation (\ref{LAE}) with initial condition $u(x,0)=f(x)$ and $x\in\mathbb{R}$ is given by $u_{\text{exact}}(x,t)=f(x+\omega t)$. We consider here its restriction to $(x,t)\in [-1,1]\times[0,1]$, assigning Dirichlet boundary conditions given by the values of $f(x+\omega t)$ at the endpoints of the spatial interval. 

As a first numerical test we set $\omega=5$ and the initial condition $u(x,0)=\sin(x)$. The exact solution, $u_{\text{exact}}(x,t)=\sin(x+5t)$, oscillates in time with frequency $\omega=5$. We solve this problem with $\Delta x=0.001 $ and $\Delta t=0.1/2^n, \,\, n=0,\ldots,11$. 

We compare here the two schemes obtained applying either the classic midpoint rule or the EF midpoint method (\ref{EFMP}) to the semidiscretization (\ref{SDLAE}). The discrete conservation laws preserved by the classic scheme are the limit for $\nu\rightarrow 0$ of (\ref{CLLAE1}) and (\ref{CLLAE2}). 

In Table~\ref{tab:LAEsin} we show the errors in the conservation laws evaluated as in (\ref{CLerr2}), the solution error, and the order of accuracy of the two schemes. The obtained results show that both methods exactly preserve the local conservation laws, and the corresponding errors are only due to accumulation of the round-offs.

The EF midpoint exactly integrates the solution of this problem, and so the error in the solution is entirely due to the space discretization for any value of $\Delta t$. By contrast, the classic midpoint method converges as a second-order method and the error in time is negligible only when $\Delta t \lesssim 10^{-4}$. This can be further seen in the graph on the left of Figure~\ref{a1}, showing a logarithmic plot of the solution errors for different values of $\Delta t$.
\begin{table}[t]
\caption{Errors in solution and conservation laws for $u(x,0)=\sin(x)$}\label{tab:LAEsin}
\small
\begingroup
\setlength{\tabcolsep}{6pt} 
\renewcommand{\arraystretch}{1.12} 
\centerline{\begin{tabular}{|c||c|c|c|c||c|c|c|c|}
\hline
&  \multicolumn{3}{r}{Classic Midpoint} && \multicolumn{3}{r}{EF Midpoint}&\\ 
\hline
$n$ & Sol err & Order & Err$_1$ & Err$_2$ & Sol err & Order & Err$_1$ & Err$_2$\\
\hline
0& 4.45e-2 &  & 8.94e-13 & 1.73e-12 & 3.43e-7 &  & 8.53e-13 & 1.58e-12\\
1& 1.10e-2 & 2.01 & 9.38e-13 & 1.86e-12 & 3.49e-7 & *** & 9.07e-13 & 1.80e-12\\
2& 2.74e-3 & 2.01 & 1.13e-12 & 2.26e-12 & 3.50e-7 & *** & 9.40e-13 & 1.84e-12\\
3& 6.85e-4 & 2.00 & 1.11e-12 & 2.22e-12 & 3.50e-7 & *** & 9.94e-13 & 1.98e-12\\
4& 1.71e-4 & 2.00 & 9.67e-13 & 1.90e-12 & 3.50e-7 & *** & 9.47e-13 & 1.84e-12\\
5& 4.31e-5 & 1.99 & 9.72e-13 & 1.96e-12 & 3.50e-7 & *** & 1.01e-12 & 2.00e-12\\
6& 1.10e-5 & 1.97 & 9.39e-13 & 1.83e-12 & 3.50e-7 & *** & 1.00e-12 & 1.88e-12\\
7& 3.02e-6 & 1.87 & 1.05e-12 & 2.11e-12 & 3.50e-7 & *** & 1.14e-12 & 2.13e-12\\
8& 1.02e-6 & *** & 1.53e-12 & 2.96e-12 & 3.50e-7 & *** & 1.48e-12 & 3.06e-12\\
9& 5.17e-7 & *** & 2.47e-12 & 5.22e-12 & 3.50e-7 & *** & 2.63e-12 & 5.46e-12\\
10& 3.92e-7 & *** & 5.10e-12 & 1.08e-11 & 3.50e-7 & *** & 4.92e-12 & 1.03e-11\\
11& 3.61e-7 & *** & 9.09e-12 & 1.95e-11 & 3.50e-7 & *** & 1.00e-11 & 2.21e-11\\
\hline
\end{tabular}}
\endgroup
\end{table}
\begin{figure}[htbp]
\begin{center}
	\begin{tikzpicture}

\begin{axis}[%
width=2.221in,
height=1.566in,
at={(0.in,0.481in)},
scale only axis,
xmode=log,
xmin=4e-05,
xmax=0.2,
xminorticks=true,
ymode=log,
ymin=1e-07,
ymax=0.2,
yminorticks=true,
title={Sol\,\,err$(\Delta t)$},
axis background/.style={fill=white},
legend style={legend cell align=left, align=left, draw=white!15!black,legend pos=north west}
]
\addplot [color=colorclassyblue, mark=square*, mark options={solid, colorclassyblue},mark size=2pt]
  table[row sep=crcr]{%
0.1	0.044490250965211\\
0.05	0.01102980420543\\
0.025	0.002742403409366\\
0.0125	0.000684800686336337\\
0.00625	0.000171394573041465\\
0.003125	4.31070587528995e-05\\
0.0015625	1.10391859789169e-05\\
0.00078125	3.02246877416683e-06\\
0.000390625	1.01830516573711e-06\\
0.0001953125	5.17265229916627e-07\\
9.765625e-05	3.92005293129475e-07\\
4.8828125e-05	3.60690371615639e-07\\
};
\addlegendentry{Classic}

\addplot [color=colorclassyorange, mark=*, mark options={solid, colorclassyorange},mark size=2pt]
  table[row sep=crcr]{%
0.1	3.4325884867272e-07\\
0.05	3.49255504586159e-07\\
0.025	3.50022100710592e-07\\
0.0125	3.50196320861707e-07\\
0.00625	3.50238024295567e-07\\
0.003125	3.50248535582029e-07\\
0.0015625	3.50251147413801e-07\\
0.00078125	3.50251805319214e-07\\
0.000390625	3.50251969604143e-07\\
0.0001953125	3.50252015062198e-07\\
9.765625e-05	3.50252034631768e-07\\
4.8828125e-05	3.50252081635567e-07\\
};
\addlegendentry{EF}

\addplot [color=black, dashed, line width=2]
  table[row sep=crcr]{%
0.1     0.015\\
0.05	0.00375\\
0.025	9.375e-4\\
0.0125	2.34375e-4\\
0.00625	5.859375e-05\\
0.003125	1.46484375e-05\\
0.0015625	3.662109375e-06\\
0.00078125	0.91552734375e-06\\
};
\addlegendentry{slope 2}

\end{axis}

\begin{axis}[%
width=2.221in,
height=1.566in,
at={(2.958in,0.481in)},
scale only axis,
xmode=log,
xmin=9.765625e-06,
xmax=0.01,
xminorticks=true,
ymode=log,
ymin=1e-07,
ymax=0.1,
yminorticks=true,
title={Sol\,\,err$(\Delta t)$},
axis background/.style={fill=white},
legend style={legend cell align=left, align=left, draw=white!15!black,legend pos=north west}
]
\addplot [color=colorclassyblue, mark=square*, mark options={solid, colorclassyblue},mark size=2pt]
  table[row sep=crcr]{%
0.005	0.0114\\
0.0025	0.00315\\
0.00125	0.00079\\
0.000625	0.000198\\
0.0003125	4.98e-05\\
0.00015625	1.27e-05\\
7.8125e-05	3.49e-06\\
3.90625e-05	1.18e-06\\
1.953125e-05	5.97e-07\\
9.765625e-06	4.52e-07\\
};
\addlegendentry{Classic}

\addplot [color=colorclassyorange, mark=*, mark options={solid, colorclassyorange},mark size=2pt]
  table[row sep=crcr]{%
0.005	0.000716\\
0.0025	0.000177\\
0.00125	4.41e-05\\
0.000625	1.11e-05\\
0.0003125	2.95e-06\\
0.00015625	9.43e-07\\
7.8125e-05	5.07e-07\\
3.90625e-05	4.26e-07\\
1.953125e-05	4.09e-07\\
9.765625e-06	4.05e-07\\
};
\addlegendentry{EF}

\addplot [color=black, dashed, line width=2]
  table[row sep=crcr]{%
0.005	0.0025\\
0.0025	0.000625\\
0.00125	0.00015625\\
0.000625	3.90625e-05\\
0.0003125	9.765625e-06\\
0.00015625	2.44140625e-06\\
};
\addlegendentry{slope 2}

\end{axis}

\end{tikzpicture}%
\end{center}
	\caption{Solution error for $u(x,0)=\sin(x)$ (left) and $u(x,0)=\log(x)\sin(x)$ (right) (logarithmic scale on both axis)}
	\label{a1}
\end{figure}
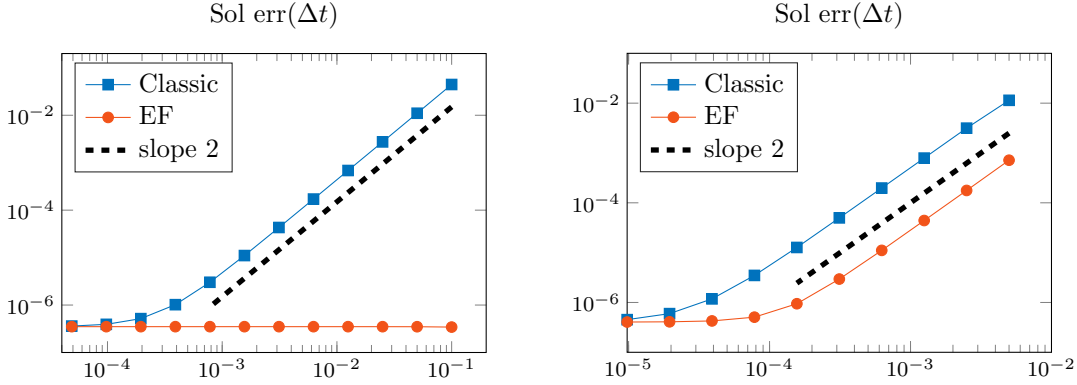

For a second numerical test, we take $\omega=50$ and the initial condition $u(x,0)=\log(x)\sin(x)$. The exact solution, $u_{\rm exact}=\log(x+50t)\sin(x+50t),$ is highly oscillatory in time, but is not exactly integrated by EF midpoint. We choose $\Delta x=0.005$ and $\Delta t=0.005/2^n,\,\,n=0,\ldots,9$. The graph on the right of Figure~\ref{a1} shows that both methods converge with order two, until the error in time is larger than the error in space. However, the EF midpoint is more accurate than the classic midpoint and it reaches the maximum accuracy with $\Delta t\simeq 10^{-4}$. In contrast, classic midpoint is equally accurate only when $\Delta t\simeq 10^{-5}$ or smaller. The results in Table~\ref{tab:LAElog} reflect the conservative properties of the schemes, and show that EF midpoint is up to 18 times more accurate than classic midpoint for the largest values of $\Delta t$.

\begin{table}[t]
\caption{Errors in solution and conservation laws for $u(x,0)=\log(x)\sin(x)$;}\label{tab:LAElog}
\small
\begingroup
\setlength{\tabcolsep}{6pt} 
\renewcommand{\arraystretch}{1.12} 
\centerline{\begin{tabular}{|c||c|c|c|c||c|c|c|c|}
\hline
&  \multicolumn{3}{r}{Classic Midpoint} && \multicolumn{3}{r}{EF Midpoint}&\\ 
\hline
$n$ & Sol err & Order & Err$_1$ & Err$_2$ & Sol err & Order & Err$_1$ & Err$_2$\\
\hline
0& 1.14e-2 &  & 3.99e-11 & 3.16e-10 & 7.16e-4 &  & 3.59e-11 & 2.81e-10\\
1& 3.15e-3 & 1.85 & 4.15e-11 & 3.33e-10 & 1.77e-4 & 2.02 & 3.62e-11 & 2.84e-10\\
2& 7.90e-4 & 1.99 & 4.72e-11 & 3.80e-10 & 4.41e-5 & 2.01 & 4.37e-11 & 3.50e-10\\
3& 1.98e-4 & 2.00 & 4.38e-11 & 3.52e-10 & 1.11e-5 & 1.98 & 4.87e-11 & 3.86e-10\\
4& 4.98e-5 & 1.99 & 4.72e-11 & 3.79e-10 & 2.95e-6 & 1.92 & 4.60e-11 & 3.73e-10\\
5& 1.27e-5 & 1.97 & 4.70e-11 & 3.93e-10 & 9.43e-7 & *** & 4.94e-11 & 3.92e-10\\
6& 3.49e-6 & 1.87 & 5.47e-11 & 4.52e-10 & 5.07e-7 & *** & 6.12e-11 & 4.96e-10\\
7& 1.18e-6 & *** & 6.69e-11 & 5.64e-10 & 4.26e-7 & *** & 7.00e-11 & 5.82e-10\\
8& 5.97e-7 & *** & 1.05e-10 & 9.09e-10 & 4.09e-7 & *** & 1.01e-10 & 8.85e-10\\
9 & 4.52e-7 & *** & 1.86e-10 & 1.77e-9 & 4.05e-7 & *** & 1.95e-10 & 1.68e-9\\
\hline
\end{tabular}}
\endgroup
\end{table}
\subsubsection*{Modified Korteweg-de Vries equation}
The mKdV equation (\ref{mKdV}) has a breather solution given by \cite{CJ} 
\begin{equation}\label{exbr}
u(x,t)=-4\frac{\eta}{\xi}\frac{\xi\cosh(\nu_2+\rho_2)\sin(\nu_1+\rho_1)+\eta\sinh(\nu_2+\rho_2)\cos(\nu_1+\rho_1)}{\cosh^2(\nu_2+\rho_2)+(\eta/\xi)^2\cos^2(\nu_1+\rho_1)},
\end{equation}
with $\xi\in\mathbb{R},$ $\eta>0,$
\begin{equation*}
\nu_1=2\xi(x+4(\xi^2-3\eta^2)t), \qquad\nu_2=2\eta(x-4(\eta^2-3\xi^2)t),
\end{equation*}
and
\begin{equation*}
\tan{\rho_1}=\frac{B\xi-A\eta}{A\xi+B\eta},\qquad \mathrm{e}^{-\rho_2}=\left|\frac{\xi}{2\eta}\right|\sqrt{\frac{A^2+B^2}{\xi^2+\eta^2}}.
\end{equation*}
We solve here the mKdV equation (\ref{mKdV}) with initial condition obtained evaluating (\ref{exbr}) at $t=0$, zero boundary conditions, and setting
$$A=3, \qquad B=\sqrt{48\xi^2-9},\qquad \eta=\xi\sqrt{3},$$
so that
$$\nu_2=2\sqrt{3}\xi x, \qquad \rho_2=0,$$
and the wave does not travel, but it only oscillates around its initial position. The frequency of oscillation can be calculated from the initial condition and it is
$$\omega=64\xi^3.$$ 

We first choose $\xi=0.7,$ so the corresponding solution oscillates with frequency $\omega=21.952.$
Figure~\ref{fig:exbr} shows the exact profile of this breather solution for $(x,t)\in[-4,4]\times[0,20]$ from three different perspectives. In particular, the view on the $x$--$t$ plane at the centre of Figure~\ref{fig:exbr} shows a high number of oscillations in the considered time window. On the right of Figure~\ref{fig:exbr}, the view on the $x$--$u$ plane highlights that the wave only moves within a compact space support, roughly the interval $ [-3,3]$, and that the superposition of all the oscillations defines a profile that is symmetric with respect to the plane $x=0$.
\begin{figure}[tbp]
\begin{center}
\includegraphics[width=.32\textwidth,height=5cm]{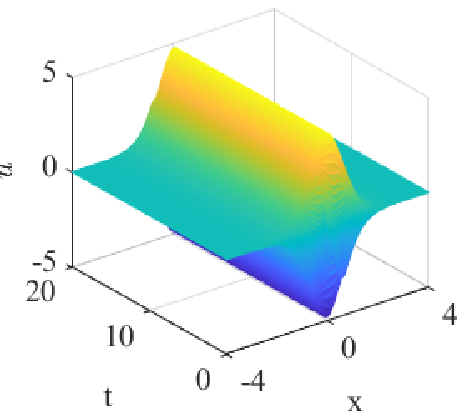}
\includegraphics[width=.32\textwidth,height=5cm]{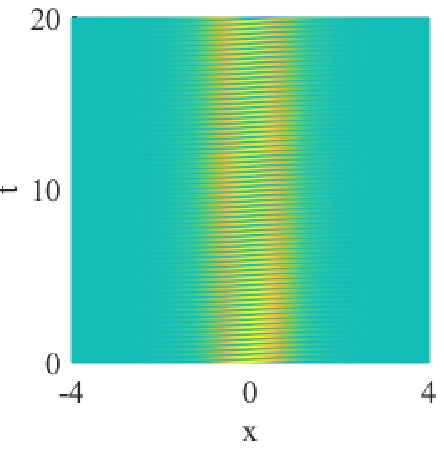}
\includegraphics[width=.32\textwidth,height=5cm]{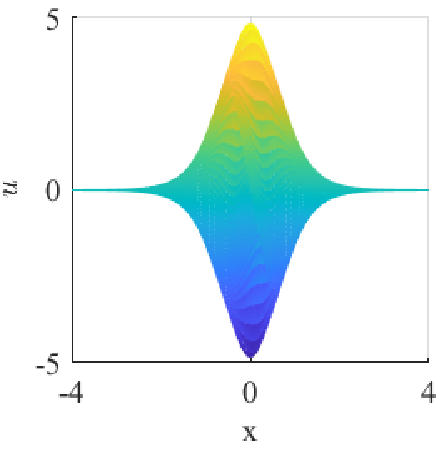}
\end{center}
	\caption{Breather solution of mKdV: 3D view (left), $x$--$t$ view (centre), $x$--$u$ view (right)}
	\label{fig:exbr}
\end{figure}

In this first numerical test we compare the solutions of the classic midpoint and of the EF midpoint applied to the semidiscretization (\ref{SDmKdV}) setting $\Delta x=0.04$ and $\Delta t=0.004$. 
\begin{figure}[tbp]
\begin{center}
\includegraphics[width=.32\textwidth,height=5cm]{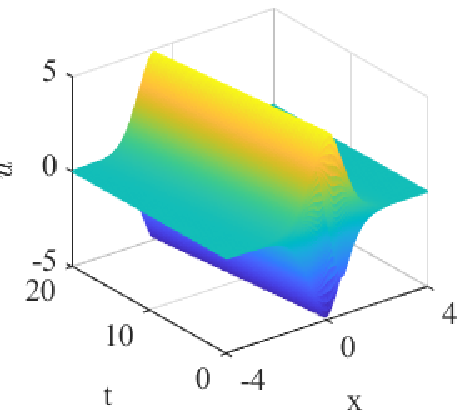}
\includegraphics[width=.32\textwidth,height=5cm]{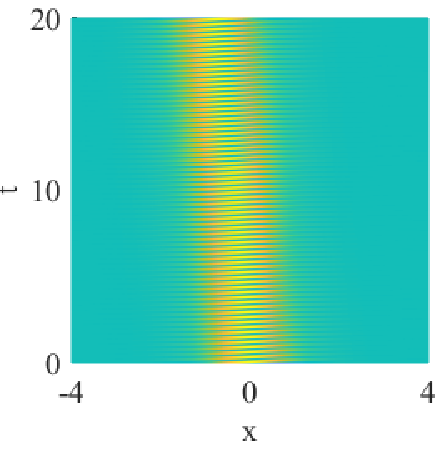}
\includegraphics[width=.32\textwidth,height=5cm]{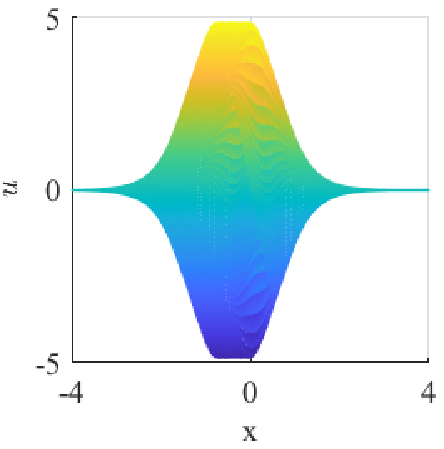}
\end{center}
	\caption{Solution of classic midpoint: 3D view (left), $x$--$t$ view (centre), $x$--$u$ view (right)}
	\label{fig:classicbr}
\end{figure} 
Figure~\ref{fig:EFbr} shows that the solution of classic midpoint travels towards negative values of $x$. Considering longer time windows, the wave reaches the left boundary and escapes out of the considered space domain. This solution is qualitatively incorrect.
\begin{figure}[tbp]
\begin{center}
\includegraphics[width=.32\textwidth,height=5cm]{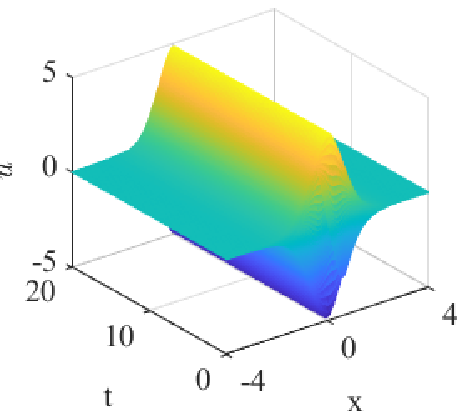}
\includegraphics[width=.32\textwidth,height=5cm]{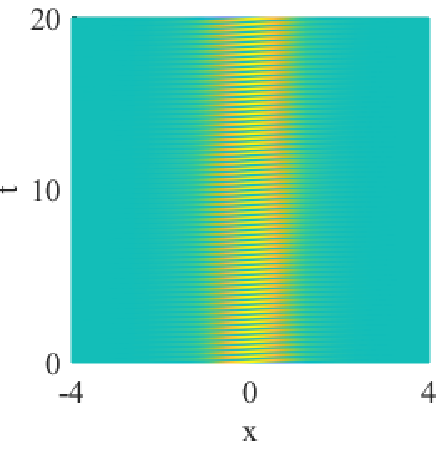}
\includegraphics[width=.32\textwidth,height=5cm]{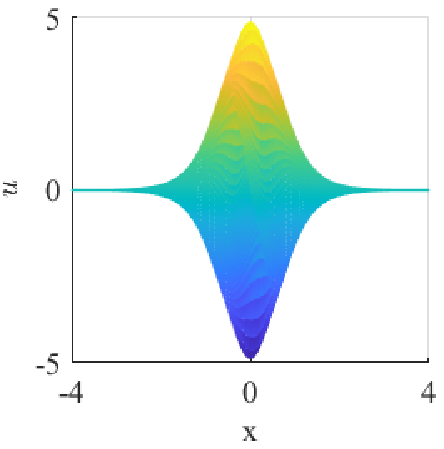}
\end{center}
	\caption{Solution of EF midpoint: 3D view (left), $x$--$t$ view (centre), $x$--$u$ view (right)}
	\label{fig:EFbr}
\end{figure} 
\begin{figure}[tbp]
\begin{center}
\includegraphics[width=.32\textwidth,height=5cm]{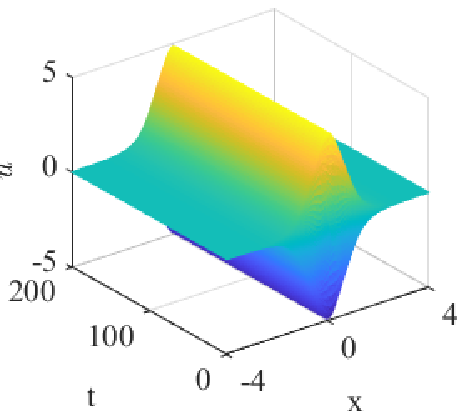}
\includegraphics[width=.32\textwidth,height=5cm]{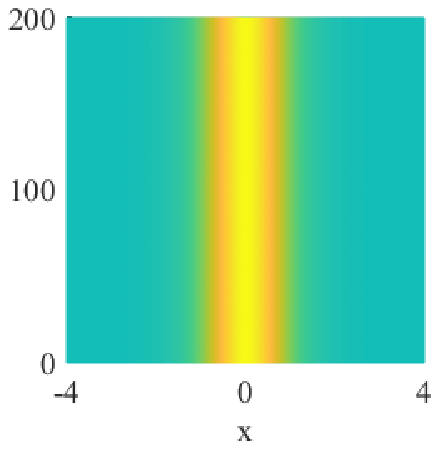}
\includegraphics[width=.32\textwidth,height=5cm]{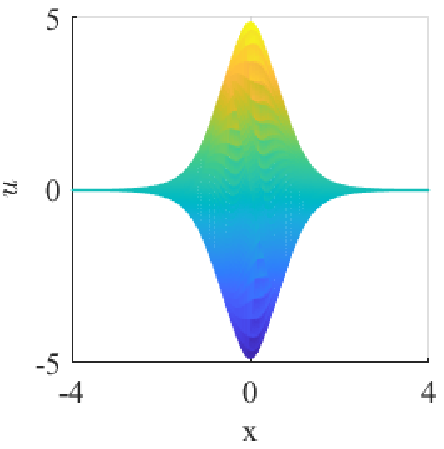}
\end{center}
	\caption{Solution of EF midpoint $t\in[0,200]$: 3D view (left), $x$--$t$ view (centre), $x$--$u$ view (right)}
	\label{fig:EFbr2}
\end{figure} 
The solution of the EF midpoint is shown in Figure~\ref{fig:EFbr}. In particular, the graph on the right shows that the motion only consists of pure oscillations and their superposition defines the correct symmetric profile. 
This method reproduces the correct qualitative behaviour of the exact solution also on longer time windows. As an example, we show in Figure~\ref{fig:EFbr2} its solution for $t\in[0,200]$.

We now set $\xi=1$ and show some quantitative comparisons. In this case the solution is a breather that oscillates with frequency $\omega=64$. We solve this problem on $(x,t)\in[-2,2]\times[0,0.2]$ with $\Delta x=0.002$ and $\Delta t=0.0032/2^n, n=0,\ldots,5$.
\begin{table}[t]
\caption{Order of convergence and error in solution and conservation laws}\label{tab:mKdV}
\small
\begingroup
\setlength{\tabcolsep}{6pt} 
\renewcommand{\arraystretch}{1.12} 
\centerline{\begin{tabular}{|c||c|c|c|c||c|c|c|c|}
\hline
&  \multicolumn{3}{r}{Classic Midpoint} && \multicolumn{3}{r}{EF Midpoint}&\\ 
\hline
$n$ & Sol err & Order & Err$_1$ & Err$_2$ & Sol err & Order & Err$_1$ & Err$_2$\\
\hline
0& 5.66e-1 &  & 2.10e-12 & 1.04e-11 & 2.48e-1 &  & 2.83e-12 & 9.60e-12\\
1& 1.62e-1 & 1.82 & 6.84e-13 & 7.70e-12 & 7.05e-2 & 1.83 & 2.06e-12 & 2.68e-11\\
2& 3.97e-2 & 2.03 & 7.98e-13 & 2.70e-12 & 1.61e-2 & 2.13 & 1.57e-12 & 3.50e-12\\
3& 8.72e-3 & 2.19 & 1.11e-12 & 5.34e-12 & 4.97e-3 & 1.70 & 9.57e-13 & 8.09e-12\\
4& 4.64e-3 & *** & 7.05e-13 & 1.19e-12 & 5.17e-3 & *** & 1.15e-12 & 2.02e-12\\
5& 5.44e-3 & *** & 5.32e-13 & 9.53e-13 & 5.65e-3 & *** & 3.47e-13 & 1.34e-12\\
\hline
\end{tabular}}
\endgroup
\end{table}
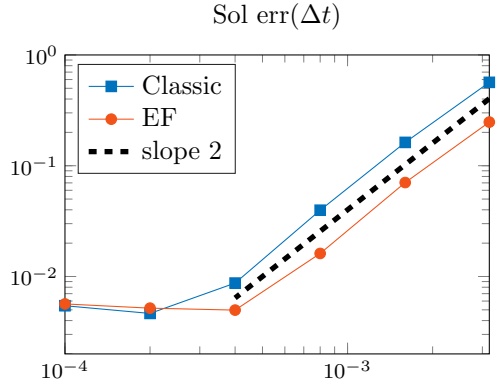
\begin{figure}[tbp]
\begin{center}
	\begin{tikzpicture}

\begin{axis}[%
width=2.221in,
height=1.566in,
at={(0.in,0.481in)},
scale only axis,
xmode=log,
xmin=0.0001,
xmax=0.0031746,
xminorticks=true,
ymode=log,
ymin=0.002,
ymax=1,
yminorticks=true,
title={Sol\,\,err$(\Delta t)$},
axis background/.style={fill=white},
legend style={legend cell align=left, align=left, draw=white!15!black,legend pos=north west}
]
\addplot [color=colorclassyblue, mark=square*, mark options={solid, colorclassyblue},mark size=2pt]
  table[row sep=crcr]{%
0.0031746	0.565945390943789\\
0.0016	0.16241155692151\\
0.0008	0.039664031381976\\
0.0004	0.008721713791316\\
0.0002	0.004640502375947\\
0.0001	0.005445157997013\\
};
\addlegendentry{Classic}

\addplot [color=colorclassyorange, mark=*, mark options={solid, colorclassyorange},mark size=2pt]
  table[row sep=crcr]{%
0.0031746	0.247646328277081\\
0.0016	0.070497006518469\\
0.0008	0.016144316602873\\
0.0004	0.004972787597417\\
0.0002	0.005178422927003\\
0.0001	0.005653037241612\\
};
\addlegendentry{EF}

\addplot [color=black, dashed, line width=2]
  table[row sep=crcr]{%
0.0031746	0.4031234064\\
0.0016	0.1024\\
0.0008	0.0256\\
0.0004	0.0064\\
};
\addlegendentry{slope 2}

\end{axis}

\end{tikzpicture}%
\end{center}
	\caption{Solution error for mKdV breather (logarithmic scale on both axis)}
	\label{fig:mKdVord}
\end{figure}

In Table~\ref{tab:mKdV} we show that the errors in the conservation laws of the two methods are of the order of the roundoffs. As before, the conservation laws satisfied by the classic midpoint method are the limit for $\nu\rightarrow 0$ of (\ref{eq:mKdVcl1}) and (\ref{eq:mKdVcl2}). Also in this case the exponentially fitted midpoint is more accurate than the classic midpoint. Both methods converge with second order until the error in time is negligible compared to the space accuracy, as is shown also in Figure~\ref{fig:mKdVord}. 
\subsubsection*{Nonlinear Schr\"odinger equation}
We solve here the nonlinear Schr\"odinger equation (\ref{NLS1}) with the initial condition yielding the following breather solution \cite{AEK}
\begin{align*}
\psi(x,t)&=\left(\frac{2\beta^2\cosh\theta+2\mathrm{i}\beta\sqrt{2-\beta^2}\sinh\theta}{2\cosh\theta-\sqrt{4-2\beta^2}\cos(\sqrt\omega\beta x)}-1\right)\sqrt\omega\mathrm{e}^{\mathrm{i}\omega t},\qquad \theta=\omega\beta\sqrt{2-\beta^2}t,\qquad\beta<\sqrt{2},\\
u(x,t)&=\operatorname{Re}(\psi),\qquad v(x,t)=\operatorname{Im}(\psi). 
\end{align*}
\begin{figure}[tbp]
\begin{center}
\includegraphics[width=.32\textwidth,height=5cm]{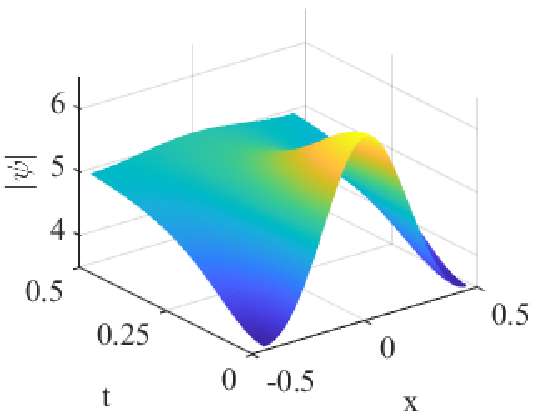}
\includegraphics[width=.32\textwidth,height=5cm]{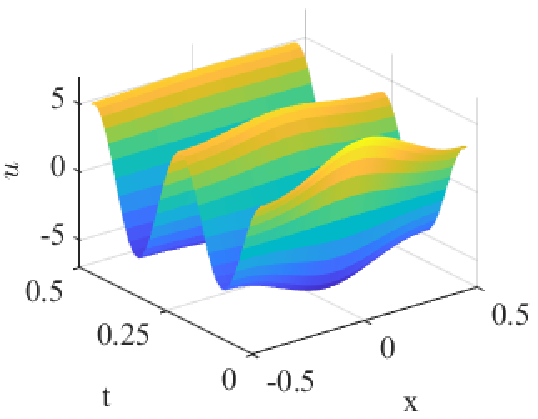}
\includegraphics[width=.32\textwidth,height=5cm]{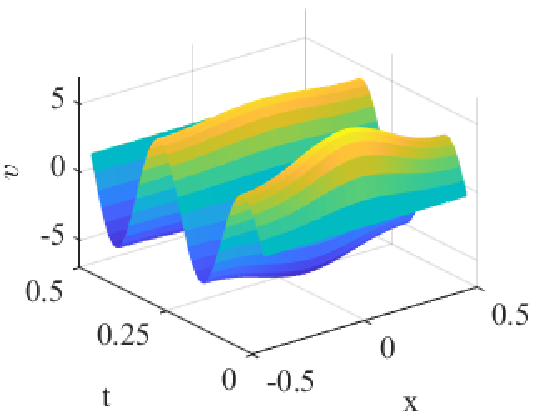}
\end{center}
	\caption{Solution of the breather problem for NLS: $|\psi(x,t)|$ (left), $u(x,t)$ (centre), $v(x,t)$ (right)}
	\label{fig:NLSbr}
\end{figure} 
\begin{table}[t]
\caption{Order of convergence and error in solution and conservation laws}\label{tab:NLS}
\small
\begingroup
\setlength{\tabcolsep}{6pt} 
\renewcommand{\arraystretch}{1.12} 
\centerline{\begin{tabular}{|c||c|c|c|c||c|c|c|c|}
\hline
&  \multicolumn{3}{r}{Classic Midpoint} && \multicolumn{3}{r}{EF Midpoint}&\\ 
\hline
$n$ & Sol err & Order & Err$_1$ & Err$_2$ & Sol err & Order & Err$_1$ & Err$_2$\\
\hline
0& 1.49e-1 &  & 3.34e-13 & 4.70e-13 & 6.72e-2 &  & 1.10e-13 & 1.48e-13\\
1& 1.70e-1 & -0.19 & 9.95e-14 & 2.18e-13 & 1.85e-2 & 1.86 & 9.59e-14 & 2.85e-13\\
2& 5.66e-2 & 1.59 & 8.17e-14 & 4.56e-14 & 4.97e-3 & 1.90 & 1.49e-13 & 2.42e-13\\
3& 1.52e-2 & 1.90 & 8.53e-14 & 8.19e-14 & 1.49e-3 & 1.74 & 9.59e-14 & 1.79e-13\\
4& 4.09e-3 & 1.89 & 9.95e-14 & 6.42e-14 & 6.23e-4 & *** & 1.28e-13 & 1.79e-13\\
5& 1.27e-3 & 1.69 & 8.88e-14 & 4.17e-14 & 4.07e-4 & *** & 8.88e-14 & 1.05e-13\\
6& 5.68e-4 & *** & 7.82e-14 & 5.57e-14 & 3.54e-4 & *** & 8.53e-14 & 3.49e-14\\
7& 3.94e-4 & *** & 8.53e-14 & 1.70e-14 & 3.41e-4 & *** & 1.03e-14 & 3.18e-14\\
\hline
\end{tabular}}
\endgroup
\end{table}
\begin{figure}[tbp]
\begin{center}
	\begin{tikzpicture}

\begin{axis}[%
width=2.221in,
height=1.566in,
at={(0.in,0.481in)},
scale only axis,
xmode=log,
xmin=0.00006,
xmax=0.01,
xminorticks=true,
ymode=log,
ymin=0.0001,
ymax=1,
yminorticks=true,
title={Sol\,\,err$(\Delta t)$},
axis background/.style={fill=white},
legend style={legend cell align=left, align=left, draw=white!15!black,legend pos=north west}
]
\addplot [color=colorclassyblue, mark=square*, mark options={solid, colorclassyblue},mark size=2pt]
  table[row sep=crcr]{%
0.01	0.149\\
0.005	0.17\\
0.0025	0.0566\\
0.00125	0.0152\\
0.000625	0.00409\\
0.0003125	0.00127\\
0.00015625	0.000568\\
0.000078125 0.0003936361246245671\\
};
\addlegendentry{Classic}

\addplot [color=colorclassyorange, mark=*, mark options={solid, colorclassyorange},mark size=2pt]
  table[row sep=crcr]{%
0.01	0.0672\\
0.005	0.0185\\
0.0025	0.00497\\
0.00125	0.00149\\
0.000625	0.000623\\
0.0003125	0.000407\\
0.00015625	0.000354\\
0.000078125 0.0003405235242281730\\
};
\addlegendentry{EF}

\addplot [color=black, dashed, line width=2]
  table[row sep=crcr]{%
0.005	0.45\\
0.0025	0.1125\\
0.00125	0.028125\\
0.000625	0.00703125\\
0.0003125	0.0017578125\\
};
\addlegendentry{slope 2}

\addplot [color=black, dashed, line width=2]
  table[row sep=crcr]{%
0.01	0.04\\
0.005	0.01\\
0.0025	0.0025\\
0.00125	0.000625\\
};

\end{axis}
\end{tikzpicture}%
\end{center}
	\caption{Solution error for NLS breather (logarithmic scale on both axis)}
	\label{fig:NLSord}
\end{figure}
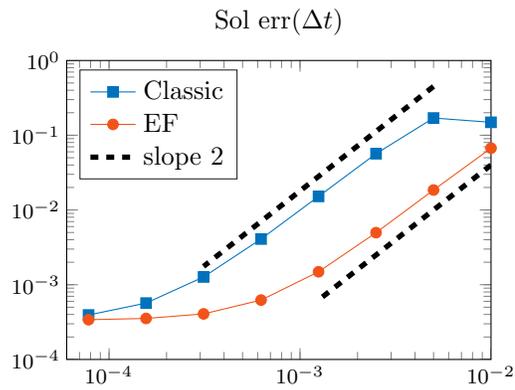
We consider the restriction of this solution to the domain $(x,t)\in [-\pi/7,\pi/7]\times[0,0.5]$ and we set $\beta=1.4$. The frequency of oscillation of $u$ and $v$ can be obtained from the initial condition and is $\omega=25$. The exact solution is plotted in Figure~\ref{fig:NLSbr} and it satisfies periodic boundary conditions. 

The numerical grids are defined with $\Delta x=2\pi/7000$, and $\Delta t=0.01/2^n$, $n=0,\ldots,6$. As shown in Table~\ref{tab:NLS} both EF midpoint and the classic midpoint preserve the conservation laws (\ref{eq:NLScl1})--(\ref{eq:NLScl2}) and their limit for $\nu\rightarrow 0,$ respectively. 

The results in Table~\ref{tab:NLS} and in Figure~\ref{fig:NLSord} show that the convergence of both methods is of the second order in time, and the error decreases with the time step until approaching the accuracy in space. However, the classic midpoint converges with the expected order only for the smaller values of $\Delta t$. The EF midpoint is up to 11 times more accurate than classic midpoint and reaches the space accuracy with larger values of $\Delta t$.
\section{Conclusions}\label{sec:concl}
In this paper, we have proved that any symplectic EFRK method preserves local conservation laws with linear or quadratic homogeneous density of suitable space discretizations of a PDE. We have also given the conditions that they need to satisfy in order to preserve conservation laws whose density is quadratic nonhomogeneous.

Space discretizations that preserve conservation laws have been obtained using the technique introduced in \cite{FoCM}. This allows to straightforwardly cope with PDEs depending on more than two independent variables similarly as done in \cite{FoCM}.  
 
On the basis of this result we have proposed exponentially fitted methods that preserve two conservation laws of the advection equation, the modified KdV equation and the system of PDEs given by the real formulation of the NLS equation. The proposed schemes are second order accurate, and higher order schemes can be similarly obtained combining higher order space discretizations with higher order EFRK methods.

Numerical tests have confirmed the conservative properties of the proposed methods as well as their convergence to the exact solution with the expected order of accuracy. Applications to problems with oscillatory solutions, such as breather waves, have shown that the proposed fitted schemes are more effective than other symplectic methods of the same order of accuracy.
\subsection*{Acknowledgements} 
This work is supported by GNCS-INDAM project and by PRIN2017-MIUR project. The authors are members of the INdAM Research group GNCS.

\end{document}